\documentclass[12pt]{article}

\newlength{\stefan}
\usepackage{amsmath, amsthm, amsfonts, makeidx}
\usepackage{multind}
\DeclareMathSymbol{\subsetneq}{\mathord}{AMSb}{"26}

\newtheorem{conj}{Conjecture}
\newtheorem{prob}{Problem}
\newtheorem{lemma}{Lemma}[section]

\newtheorem{theorem}[lemma]{Theorem}

\newtheorem{corollary}[lemma]{Corollary}
\theoremstyle{definition}

\newtheorem{definition}[lemma]{Definition}

\newtheorem{remark}[lemma]{Remark}

\newcommand{\lp}{\longrightarrow}

\newcommand{\mb}{\mathbb}

\newcommand{\G}{\mathcal{G}}
\newcommand{\F}{\mathbb{F}}

\newcommand{\E}{\mathcal{E}}

\newcommand{\C}{\mb{C}}

\newcommand{\R}{\mb{R}}
\newcommand{\Z}{\mb{Z}}
\newcommand{\N}{\mb{N}}
\newcommand{\Q}{\mb{Q}}

\newcommand{\desda}{\Longleftrightarrow}
\newcommand{\Aff}{\textup{Aff}}

\newcommand{\DER}{\textup{DER}}

\renewcommand{\ker}{\operatorname{ker}}

\renewcommand{\deg}{\operatorname{deg}}

\newcommand{\LND}{\textup{LND}}
\newcommand{\LFD}{\textup{LFD}}

\newcommand{\K}{k}

\newcommand{\GA}{\textup{GA}}
\newcommand{\GL}{\textup{GL}}

\newcommand{\kar}{\operatorname{char}}
\newcommand{\TA}{\textup{T}}

\title{The Nagata automorphism is shifted linearizable}
\author{\begin{tabular}{ll}
Stefan Maubach\footnote{Funded by Veni-grant of council for the
physical sciences, Netherlands Organisation for scientific research (NWO).
Partially funded by the Mathematisches
Forschungsinstitut Oberwolfach as an Oberwolfach-Leibniz-Fellow.}&Pierre-Marie Poloni\\
\ \\
\small
Radboud University Nijmegen&\small Institut Math\'ematiques de Bourgogne\\
\small Postbus 9010, 6500 GL Nijmegen &\small 9, avenue Alain Savary, BP 47870\\  
\small The Netherlands&\small 21078 Dijon Cedex, France\\
\small s.maubach@math.ru.nl& \small ppoloni@u-bourgogne.fr
\end{tabular}}

\begin{document}

\maketitle

\begin{abstract}
A polynomial automorphism $F$ is called {\em shifted linearizable} if there exists a linear map $L$ such that $LF$ is linearizable.
We prove that the Nagata automorphism $N:=(X-Y\Delta -Z\Delta^2,Y+Z\Delta, Z)$ where $\Delta=XZ+Y^2$ is shifted linearizable.
More precisely, defining $L_{(a,b,c)}$ as the diagonal linear map having $a,b,c$ on its diagonal, we prove that if
$ac=b^2$, then $L_{(a,b,c)}N$ is linearizable if and only if $bc\not = 1$.
We do this as part of a significantly larger theory: for example, any exponent of a homogeneous locally finite derivation is shifted linearizable.
We pose the conjecture that the group generated by the linearizable automorphisms may generate the group of automorphisms,
and explain why this is a natural question.
\end{abstract}

\section{Preliminaries}
\label{pre}

\subsection{Introduction}

One of the main problems in affine algebraic geometry is to
understand the polynomial automorphism group of affine spaces. In
particular, it would be very useful to find some generators of
these groups. The case of dimension one is easy : every
automorphism of the affine line is indeed affine. (For a
polynomial map, to be affine  means to be of degree 1.)

In dimension two, the situation is well known too. The Jung-van
der Kulk-theorem asserts that the automorphism group of the affine
plane is generated by affine and de Joncqui\`ere subgroups \cite{Jung, Kulk}.
 Therefore, every automorphism of $\mathbb{A}^2$ is called
\emph{tame}.

The case of dimension 3 is still open. Recently, Umirbaev and
Shestakov solved in \cite{SU04, SU04a}, the thirty years old {\em
tame generators problem} by proving that some automorphism of
$\mathbb{C}^3$ are not tame and in particular that the famous
Nagata map is non tame.

Actually, there are several  candidate generator sets for the
automorphism group of $\mathbb{A}^n$ (see section \ref{final}).

Nevertheless, from a ``geometric point of view'', it is important to
find generators which do not depend on choice of coordinates.
Related, finding normal subgroups of the automorphism group, is important in itself (and almost the same question, actually).
 Notice
that, since a non tame automorphism may be conjugate to a tame
one (theorem \ref{NagataIsInLIN} gives such an example), the notion of tame automorphism is not a relevant geometric
notion.

Therefore, it seems natural to define \emph{tamizable}
automorphisms, i.e. automorphisms which are conjugate to a tame
one. In particular, it leads us to the following questions :

\begin{enumerate}
\item Is the Nagata automorphism tamizable?

\item Are all automorphisms of $\mathbb{C}^3$ tamizable?
\end{enumerate}

Note that if the answer to the first question is negative, then it
will be very difficult to prove it. (The concept of degree is not
invariant under conjugation, and so, the proof of
Umirbaev-Shestakov does not  give ideas for this.)

In this paper, we will investigate the second question and study
what consequences a positive answer will give. It will lead us to
consider the subgroup
$\textup{GLIN}_n(\C)\subseteq\textup{GA}_n(\C)$ generated by
linearizable automorphisms. It turns out that this group contains
all tame automorphisms, and, more surprising, that the Nagata
automorphism belongs to $\textup{GLIN}_3(\C)$.

More precisely, we will show that {\em ``twice Nagata''} is even
linearizable! ``Twice Nagata'' stands for the map $(2I)\circ N$,
i.e. each component of the Nagata automorphism multiplied by 2. Then
\[ N^{\frac{4}{3}} (2N)   N^{\frac{-4}{3}}=2I\]
as explained in theorem \ref{NagataIsInLIN}.
In fact, we will
prove that  if $D$ is a homogeneous locally finite derivation on
$\C^{[n]}$, then there exists $s\in\C^*$ such that
$s\exp(D)=(sI)\circ\exp(D)$ is linearizable. We say that $\exp(D)$
is {\em shifted linearizable}.

In the analytic realm, this is a known local fact, due to the
Poincar\'e-Siegel theorem (see \cite{A83}, chapter 5, or 8.3.1. of
\cite{Essenboek}). Roughly, this theorem states that for almost
all $s\in \C^*$, and analytic map $F$ satisfying $F(0)=0$, $sF$ is
holomorphically linearizable locally around 0. This theorem was
the starting point of a very interesting story\footnote{To save
space we have to refer to \cite{Essenboek} page 185 and beyond, or
the review \cite{AMS}} about the (negative) solution of the
Markus-Yamabe conjecture and its link to the Jacobian conjecture,
see \cite{CEGHM94,E,EH96}. One of the conjectures which was posed
and killed ``along the way'' of this story was Meister's
Linearization conjecture (see page 186 of \cite{Essenboek}
or \cite{DMZ95}). However, the current article can be seen as a partial positive answer to a generalized Meister's conjecture -- in fact, to such an extent that we revive a reformulate Meister's conjecture:\\
\ \\
{\bf Meister's Linearization Problem:} For which $F\in \GA_n(\C)$ does there exist some $s\in \C^*$ such that $sF$ is linearizable?\\
\ \\
This article is organized as follows. In section {\em \ref{pre}:
Preliminaries} we define notations and mention well-known facts on
derivations. In section {\em \ref{main}: Shifted linearizability}
we show how to shift-linearize homogeneous derivations. In section
{\em \ref{nagata}: When is Nagata shifted linearizable?} we use
the previous section on Nagata's map as an example, and explain
exactly for which shifts it is linearizable and when it isn't. (We
will prove that $sN$ is linearizable  if and only if $s\not =
1,-1$.) In the last section \ref{final} we will discuss how the
results of this article influence the current conjectures on
generators of $\GA_n(\C)$.

\subsection{Notations and definitions}
\label{notation}

Let $R$ be a commutative ring with one. (In this article, $R$ will be $\C$ almost exclusively.)
$R^{[n]}$ will denote the polynomial ring in $n$ variables over $R$.
$\GA_n(R)$ will denote the group of polynomial automorphisms on $R^{[n]}$.
We will denote $I$ for the identity map.
$\partial_X$ ($\partial_Y,\partial_Z,\ldots$) will denote the derivative
to the variable $X$  ($Y,Z,\ldots$).

An $R$-derivation (or simply derivation if no confusion is possible) on an $R$-algebra $A$ is an $R$-linear map $D:A\lp A$ that
satisfies the Leibniz rule $D(ab)=aD(b)+bD(a)$ for each $a,b\in A$.
The set of $R$-derivations (or derivations) on $A$ is denoted by $\DER_R(A)$ (or $\DER(A)$).
The set of $R$-derivations on $R^{[n]}$ is denoted by $\DER_n(R)$. $\DER(A)$
forms a Lie algebra, as any two derivations $D,E$ the map $[D,E]:=DE-ED$ is again a derivation, as can be
easily checked.
A {\em locally nilpotent derivation} is a derivation $D$ for which each $a\in A$ one finds an $m\in \N$ such that $D^m(a)=0$. For example: $D=\partial_X$ on $\C[X]$.
If $R=\K$, a field, we define a {\em locally finite derivation} as a derivation $D$ for which each $a\in A$ the $\K$-span of $a,D(a),D^2(a),\ldots$ is
finite dimensional. For example:
$D=(X+1)\partial_X$ on $\C[X]$. We use $\LND_n(\K), \LFD_n(\K)$ for the sets of locally nilpotent resp.
locally finite derivations on $\K^{[n]}$.

If $D$ is a derivation on a ring $A$ containing $\Q$, then one can define the map $\exp(TD):A[[T]]\lp A[[T]]$ as the map sending $f$ to
$\sum_{i=0}^{\infty} \frac{T^i}{i!}D^i(f)$.
It is an automorphism of $A[[T]]$, and its inverse is $\exp(-TD)$.
In case $D$ is locally nilpotent, the map $\exp(D):A\lp A$ is well-defined and again an automorphism (with inverse $\exp(-D)$). In case $D$ is locally finite,
one cannot always define the exponential map. For one, the field $\K$ must satisfy ``$a\in \K$ then $\sum_{i=0}^{\infty}\frac{a}{i!}\in \K$''.
We will only take exponents of locally finite derivations in case $\K=\C$.

We define the derivation $\delta$ on $\C[X,Y,Z]$ and the
polynomial $\Delta\in \C[X,Y,Z]$ by
$\delta:=-2Y\partial_X+Z\partial_Y$, and $\Delta:=XZ+Y^2$.
$\Delta\delta$ will be the Nagata derivation, and  $N$ will denote
the Nagata automorphism:
\[ N=\exp(\Delta\delta)=(X-2Y\Delta-Z\Delta^2, Y+Z\Delta, Z). \]
If $\lambda\in \C$, we denote $N^{\lambda}$ the following
automorphism of $\C[X,Y,Z]$: \[N^{\lambda}:=\exp(\lambda \Delta
\delta)=(X-\lambda 2Y\Delta-\frac{1}{2}\lambda^2Z\Delta^2,
Y+\lambda Z\Delta, Z).\].

Note that one can also use this formula to define $N^{\lambda}$ as
an automorphism of $k[X,Y,Z]$ for any field of characteristic
$\kar(k)\not = 2$ and any $\lambda \in k$.

\subsection{A basic result}

\begin{lemma} \label{derivation} Let $D\in \LND(\C^{[n]})$, and $p\in \C^{[n]}$, $p\not =0$.
If $\exp(D)(p)=\lambda p$, then $\lambda =1$, and $D(p)=0$.\\
\end{lemma}

\begin{proof}
Let $q\in \N$ such that $D^q(p)\not =0, D^{q+1}(p)=0$. Then $D^q(p)=D^q(\exp(D))(p)=D^q(\lambda p)=\lambda D^q(p)$ hence $\lambda=1$.
Assume $q\geq 1$. Now
$0=D^{q-1}(0)=D^{q-1}(\exp(D)(p)-p)=
D^{q-1}(\sum_{i=1}^q (i!)^{-1}D^i(p))=
D^q(p)$. Contradiction, hence $q=0$.
\end{proof}

\section{Shifted linearizability}
\label{main}

\subsection{Definition}


We will define $F\in \GA_n(\C)$ to be {\em shifted linearizable}
if there exists a linear map $L\in \GL_n(\C)$ such that $LF$ is
linearizable, i.e. exist $G\in \GA_n(\C)$ and $L'\in \GL_n(\C)$
such that $G^{-1}LFG=L'$.


A special case is if $sF$ is linearizable, where $s\in \C^*$. In
this case $L=sI$.

\subsection{Noncommuting derivations forming a Lie algebra}

\label{SS2.2}

Well-known is that any two-dimensional Lie algebra over $\C$ which is non-commutative is essentially the Lie algebra
$\C X+\C Y$ where $[X,Y]=X$.  This Lie algebra turns up in this section as the sub Lie algebra
of $\DER_n(\C)$ generated by two derivations $D,E$ satisfying $[E,D]=D$.

\begin{lemma}\label{E.1}
Let $D,E$ be derivations, $E\in \LFD_n(\C)$, such that $[E,D]=\alpha D$ where $\alpha \in \C$.
Then
\[ \exp(\beta E) D=e^{\alpha \beta} D\exp(\beta E) \]
for any $\beta\in\C$.
\end{lemma}

The assumption $E\in \LFD_n(\C)$ is only here to make sure that $\exp(\beta E)$ is well-defined. However, if one interprets $\beta$ as a variable in the
ring $\C^{[n]}[[\beta]]$, this assumption is not necessary.

\begin{proof}
One can compute this directly, but easier is to use the well-known formulae
\[ \exp(A)B\exp(-A)=\exp([A,-])\circ B  \]
where $A,B$ are elements of a Lie algebra. In this case, conjugating $D$ by $  \exp(\beta E)$ yields
\[ (\exp[\beta E,-]) \circ D =I+\beta[E,D]+\frac{\beta^2}{2!}[E,[E,D]]+\ldots = D+\beta \alpha D + \frac{(\beta \alpha)^2}{2!}D+\ldots = e^{\beta \alpha}D.\]
This concludes the proof.
\end{proof}

\begin{corollary}\label{E.2}
Let $D,E\in \LFD_n(\C)$ and suppose $[D,E]=\alpha D$ where $\alpha\in \C$. Then for any $\beta,\lambda\in \C$ we have
\[\exp(\beta E)\exp(\lambda D)=\exp(e^{\alpha\beta}\lambda D)\exp(\beta E).\]
In particular, if $\alpha\beta\in 2\pi i \Z$ then $\exp(\beta E)$ and $\exp(\lambda D)$ commute for each $\lambda \in \C$.
\end{corollary}

\begin{proof}
Follows from lemma \ref{E.1}, which one can use to show that
\[ \exp(\beta E) D^i=(e^{\alpha \beta})^i D^i\exp(\beta E). \]
\end{proof}

\begin{corollary}\label{E.3}
Let $D,E\in \LFD_n(\C)$ and suppose $[D,E]=\alpha D$ where $\alpha\in \C$. Then for any $\beta,\lambda\in \C$, $\exp(\beta E)\exp(\lambda D)$ is conjugate to
$\exp(\beta E)$ as long as $\alpha\beta\not \in  2\pi i\Z$. In particular,
\[\exp(-\mu D)(\exp(\beta E)\exp(\lambda D))\exp(\mu D) = \exp(\beta E) \]
where $\mu= \lambda(e^{-\alpha\beta}-1)^{-1}$.
\end{corollary}

\begin{proof}
>From corollary \ref{E.2}, we replace $\lambda$ by $-e^{-\alpha\beta}\mu$ to get
\[\exp(\beta E)\exp(-e^{-\alpha\beta}\mu D)=\exp(-\mu D)\exp(\beta E).\]
This means that
\[\exp(-\mu D)(\exp(\beta E)\exp(\lambda D))\exp(\mu D) = \exp(\beta E)\exp((-e^{-\alpha\beta}\mu+\lambda+\mu) D). \]
Setting $-e^{-\alpha\beta}\mu+\lambda+\mu=0$ yields $\mu=\lambda(e^{-\alpha\beta}-1)^{-1}$.
\end{proof}

\subsection{Linearizing exponents of monomial homogeneous derivations}

As an application of the previous section we will show how to shift-linearize exponents of monomial homogeneous derivations.

A grading $\deg$ on $\C^{[n]}$ is called {\em monomial} if each
monomial (or equivalently, each variable $X_i$) is homogeneous.  It is the typical grading one puts on $\C^{[n]}$: one assigns weights to the variables $X_i$.
In fact, let us state
\[ w_i:=\deg(X_i)\]
for this article.
A homogeneous derivation is a derivation that sends homogeneous elements to homogeneous elements -- in this article, homogeneous w.r.t. some {\em monomial} grading.
It is not too difficult to check that there exists a unique $k$ such that a homogeneous element of degree $d$ is sent to a homogeneous element of degree $d+k$ or to the zero
element.
We say that $D$ is {\em homogeneous of degree $k$}.

(Above, we did not specify in which set $w_i, d,k$ are. Typical is to have them in $\N, \Z$, or even $\R$, and that is what we think of in this article. It is however possible to choose a grading which takes
values in a group, i.e. a group grading. The above explanation makes sense for this.)

For this section, define the derivation associated to $\deg$ as $E:=E_{\deg}:=\sum_{i=1}^n w_i X_i\partial_{X_i}$. ($E$ stands for Euler derivation.)
The goal of this section is to prove the following theorem:

\begin{theorem}\label{E.5}
If $D\in \LFD_n(\C)$ is homogeneous of degree $k\not = 0$ w.r.t. a monomial grading,  then $\exp(D)$ is shifted linearizable.
\end{theorem}

\begin{proof} Follows immediately from corollary \ref{E.3} and lemma \ref{hom} below, and the observation that
 $\exp(E)$ is a linear map: the diagonal map $(e^{w_1}X_1,\ldots, e^{w_n}X_n)$.
\end{proof}

\begin{lemma} \label{hom}
Let $D$ be a homogeneous derivation of degree $k$ with respect to a monomial grading $\deg$.
Then $[E_{\deg},D]=kD$. In particular, if $k=0$, then $D$ and $E_{\deg}$ commute.
\end{lemma}

\begin{proof}
Let $M:=X_1^{v_1}\cdots X_n^{v_n}$ ($v_i\in \N$) be an arbitrary monomial of degree $d$. Then $\deg(D(M))=d+k$, and
$E(M)=\sum_{i=1}^n v_iw_i M =d M$. Similarly $E(DM)=(d+k)D(M)$. Now one can see that
\[ \begin{array}{rl}
[E,D](M)=E(D(M))-D(E(M))=&(d+k) DM- D(dM)\\
=&kD(M).
\end{array}\]
Thus, $[E,D]=kD$.
\end{proof}

\section{When is Nagata shifted linearizable?}
\label{nagata}

\subsection{Using Nagata's  homogeneousness}

For the rest of this section, $D:=\Delta \delta$ will be Nagata's derivation.
The Nagata derivation $D$ is homogeneous to several monomial gradings. The set of monomial gradings form a vector space (for if $\deg_1, \deg_2$ are
the associated
degree functions, then $\deg_1+\deg_2$ and $c\deg_1$ where $c\in \C$ are degree functions associated to a grading too).

Let us explain how we find all homogeneous derivations for the Nagata derivation. More details on such procedure one can find in
\cite{Mau00a} and pages 228-234 of \cite{Essenboek}, where it is explained how to do this to prove that Robert's derivation is a counterexample to Hilbert's 14th problem.
First, notice that the variables $X,Y,Z$ are homogeneous, lets say of degree $s,t,u$ respectively. These values determine the degree function $\deg$ completely.
Now we need to satisfy the following two requirements:\\
(1) $D(X),D(Y),D(Z)$ all are homogeneous,\\
(2) $\deg(D(X))-\deg(X)=\deg(D(Y))-\deg(Y)$. (This condition comes from the fact that there should be a constant $d$ (which is the degree of $D$) for which we
have:
any homogeneous $H$ is homogeneous of degree n, then $D(H)$ is of degree $n+d$ or $D(H)=0$.)

>From $D(X),D(Y)$ homogeneous we derive that $\Delta=XZ+Y^2$ is homogeneous, and thus $s+u=2t$. Now (1) is satisfied. From (2) we get that
$s-(t+2t)=t-(u+2t)$ which yields the exact same equation $s+u=2t$. Thus $[\deg(X),\deg(Y),\deg(Z)]=[s,t,2t-s]$ and the derivation is of degree $3t-s$.
The degree function is associated with the (semisimple) derivation $E:=sX\partial_X+tY\partial_Y+(2t-s)Z\partial_Z$ and the diagonal linear map
$\exp(E):=(e^sX,e^tY, e^{2t-s}Z)$.

Thus, for the Nagata derivation, the set of gradings for which it is homogeneous, is two dimensional.
A possible basis is $\{\deg_1,\deg_2\}$ where
\[ \begin{array}{c}
\deg_1((X,Y,Z))=(1,0,-1),\\
 \deg_2((X,Y,Z))=(0,1,2) .
 \end{array}\]
The degree function $\deg_1$ corresponds to the (semisimple) derivation $E_1:=X\partial_X-Z\partial_Z$, where $\deg_2$ corresponds to $E_2:=Y\partial_Y+2Z\partial_Z$.
Any degree function $\deg=s\deg_1+t\deg_2$ ($s,t\in\C$) which is a linear combination of $\deg_1,\deg_2$ corresponds to $E:=sE_1+tE_2$.
The linear map corresponding to the linear combination $sE_1+tE_2$
is $L_{s,t}:=(e^{s}X,e^{t}Y,e^{2t-s}Z)$. The set of these maps is exactly the set
\[ \mathcal{L}:=\{(aX,bY,cZ)~|~ ac=b^2, abc\not = 0\}. \]
Thus we have proven the following lemma:

\begin{lemma} \label{E.10}
$D$ is of degree 0 with respect to  $\deg=s\deg_1+t\deg_2$ if and only if  $s=3t$.
\end{lemma}

\begin{definition}\label{E.11}
Let us define $L_b:=(b^3X, bY, b^{-1}Z)$, and $\mathcal{L}_0$ as the set $\{L_b~|~b\not =0\}$.
Note that
\[ \mathcal{L}_0:=\{(aX,bY,cZ)~|~ ac=b^2, bc=1\}. \]
\end{definition}

\subsection{Explicit formulae for shifted linearizableness of the Nagata map}

One can use corollary \ref{E.3}, theorem \ref{E.5},  and results
of the previous section, to immediately get formulas for many
linear maps $L\in\mathcal{L}$  which satisfy $LN$ is linearizable.
However, let us give the following formulas, which are slightly
more elegant, and can be easily checked directly. Moreover, they
work for any field $k$ of characteristic $\kar(k)\not = 2$ (see
subsection \ref{notation}). To be clear, for this section, we are
working over a field $k$ satisfying $\kar(k)\not =2$.
Write $L_{(a,b,c)}:=(aX,bY,cZ)$ where $ac=b^2$. The following formulas can be easily checked:\\
\begin{itemize}
\item $L_{(a,b,c)}^{-1}DL_{(a,b,c)}=(bc)^{-1}D$, which implies
\item $L_{(a,b,c)}^{-1}\exp(\lambda D)L_{(a,b,c)}=\exp(b^{-1}c^{-1}\lambda D)$, which implies
\item $N^{\lambda}L_{(a,b,c)}=L_{(a,b,c)}N^{b^{-1}c^{-1}\lambda}$.
\end{itemize}

Using the latter equation,  the following is easy:

\begin{theorem} \label{NagataIsInLIN}
 Let $a,b,c\in k^*$, $ac=b^2$, and $bc\not = 1$. Then
$L_{(a,b,c)}N^{\lambda}$ is conjugate to $L_{(a,b,c)}$. In particular,
choosing $\mu = bc\lambda(1-bc)^{-1}$, we have
\[ N^{-\mu} ( L_{(a,b,c)}N^{\lambda} ) N^{\mu} =L_{(a,b,c)} .\]
\end{theorem}

The particular case that $L$ is a multiple of the identity, gives the formulae for $s\in k^*$:
\[ N^{-\frac{s^2\lambda}{1-s^2}} ( sN^{\lambda} ) N^{\frac{s^2\lambda}{1-s^2}} =sI.\]
This gives the formula for $s=2, \lambda=1$ from the introduction. In the same introduction it was announced that we can linearize for any $s\not=1,-1$,
which indeed follows from this.

\begin{remark} \label{NonTameLinearizable} Maps $L_{(a,b,c)}N$ as in theorem \ref{NagataIsInLIN}, are non-tame (provided $\kar(k)=0$) but linearizable (and in particular, tamizable).
\end{remark}

\subsection{The non-linearizable case}

\label{NLN}

We will now consider what happens if the grading of the previous section is such that $D$ is homogeneous of degree 0.
By lemma \ref{hom} this means that $E$ commutes with $D$, and hence also $\exp(E)$ commutes with $\exp(D)$.
By lemma \ref{E.10} and definition \ref{E.11}, we can say $\exp(E)\in \mathcal{L}_0$, i.e. $\exp(E)=L_b=(b^{3}X,bY,b^{-1}Z)$ for some $b\in \C^*$.
Now there are several ways of showing that $L_bN^{\lambda}$ is not linearizable, we will use invariants.

\begin{definition}
Let $\varphi\in \GA_n(\C)$, $\lambda\in \C$. Then $\E_{\mu}(\varphi):=\{p\in \C[X_1,\ldots,X_n] ~|~ \varphi(p)=\mu p\}$
is defined as the eigenspace of $\varphi$ with respect to $\mu$.
\end{definition}

If $L_bN$ is linearizable, it will be linearizable to $L_b$ (as the linear part is equal to $L_b$).
We will show that $\E_1(L_bN^{\lambda})$ and $\E_1(L_b)$ are so different that they contradict the following property:

\begin{lemma} \label{eigenspace}
If $\varphi, \tilde{\varphi} \in \GA_n(\C)$ are conjugate (i.e. there exists $\sigma\in \GA_n(\C)$ such that
$\tilde{\varphi}=\sigma^{-1}\varphi\sigma$) then $\E_{\mu}(\varphi)$ and  $\E_{\mu}(\tilde{\varphi})$ are isomorphic
(in fact,  $\E_{\mu}(\tilde{\varphi})=\sigma^{-1}(\E_{\mu}(\varphi))$).
\end{lemma}

\begin{proof}
\[ \begin{array}{rl}
p\in \E_{\mu}(\varphi) &\desda\\
\varphi(p)=\mu p &\desda\\
\varphi \sigma \sigma^{-1} (p) = \mu p &\desda\\
\sigma^{-1}\varphi \sigma \sigma^{-1} (p) = \mu \sigma^{-1}(p)&\desda \\
\sigma^{-1}(p)\in \E_{\mu}(\sigma^{-1}\varphi \sigma).
\end{array} \]
\end{proof}

\begin{lemma} \label{eigenvalue}Let $b\in \C^*$ be no root of unity, $\lambda\in \C^*$.
Then $L_bN^{\lambda}(p)=p$ for some $p\in \C[X,Y,Z]$ if and only if $p\in \C[Z^2\Delta]$.
\end{lemma}

\begin{corollary}  \label{end} $L_bN^{\lambda}$ is not linearizable for any $b,\lambda\in \C^*$.
\end{corollary}

\begin{proof}[Proof of lemma \ref{eigenvalue}]
Give weights $w(X)=3,w(Y)=1,w(Z)=-1$ making $A:=\C[X,Y,Z]$ into a graded algebra $\oplus_{n\in \Z} A_n$. $D$ and $L_b$ are homogeneous:
$L_b(A_n)=A_n$ and  $D(A_n)\subseteq A_n$. Because of the latter, $N^{\lambda}$ is homogeneous too: $N^{\lambda}(A_n)\subseteq A_n$
(actually ``$=$'' since it is an automorphism). Hence $L_bN^{\lambda}(A_n)=A_n$. For $L_b$ we have $L_b(p)=b^{n}p$ if $p\in A_n$.

We want to find all $p$ such that $L_bN^{\lambda}(p)=p$. It suffices to classify all such $p$ which are homogeneous. Let $n=\deg(p)$.
It now must hold that $N^{\lambda}(p)=L_b^{-1}(p)=b^{-n}p$.
Because of lemma \ref{derivation}, we have
$b^{-n}=1$ and $p\in \ker{\Delta \delta}$. Hence, since $b$ is no root of unity we get $n=0$, and so $p\in \ker{\Delta \delta}\cap A_0=\C[\Delta,Z]\cap A_0$.
Since $\Delta\in A_2$ and $Z\in A_{-1}$, we get that $p\in  \C[\Delta Z^2]$.
It is easy to check that such $p$ indeed satisfy $L_bN^{\lambda}(p)=p$.
\end{proof}

\begin{proof}[Proof of corollary \ref{end}]
Assume $L_bN^{\lambda}$ is linearizable.
We split the proof in two cases: \\
{\em Let $b^m=1$ for some $m\in \N^*$:}
Thus there exists some $\varphi\in \GA_3(\C)$ such that $\varphi^{-1}L_bN^{\lambda}\varphi=L_b$.
Thus $I=(L_b)^m=(\varphi^{-1}L_bN^{\lambda}\varphi)^m=\varphi^{-1}N^{m\lambda}\varphi$.
Thus $N^{m\lambda}$ must be the identity, which implies that $m=0$, contradiction.   \\
{\em $b$ is no root of unity:}
By lemma \ref{eigenvalue} $\mathcal{E}_1(L_bN^{\lambda})$ is isomorphic to $\C[\Delta Z^2]$.
By lemma \ref{eigenspace}, we must have that $E_1(L_bN^{\lambda})$ is isomorphic to $E_1(L_b)$.
However, their transcendence degrees differ.
\end{proof}



\section{Generators of $\GA_n(\C)$ and conjectures}
\label{final}

\subsection{Tamizable automorphisms}

The following definition and the problems \ref{C1} and \ref{C2} were given to us by  A. Dubouloz.

\begin{definition} A polynomial automorphism $\varphi$ is called {\em tamizable} if there exists a
polynomial automorphism $\sigma$ such that $\sigma^{-1}\varphi\sigma$ is tame (in analog to linearizable and triangularizable).
\end{definition}

Now let us repeat the conjectures from the introduction:

\begin{prob}  \label{C1} Is $N$ tamizable? (Is every automorphism of $\C^{[3]}$ tamizable?)
\end{prob}

\begin{prob} \label{C2} Is $N$ tamizable by conjugation of an element of $\GA_2(\C[Z])$?
\end{prob}

Connected to this, we also mention the following problem, which we took from \cite[p.120]{Fre06}:

\begin{prob} \label{C-1}Every tame $\G_a$-action on $\C^3$ is conjugate to a triangular action.
\end{prob}

Note that the problems \ref{C1} and \ref{C-1} cannot both be true.

\subsection{Known conjectures}

Since the ``tame generators conjecture''  (which hardly anyone believed because of the automorphism $N$)
was disproved by Umirbaev-Shestakov in \cite{SU04, SU04a} (and also before this feat was accomplished),
several new conjectures have been made of ``understandable'' sets which could generate all of $\GA_n(\C)$ for any $n$.
We will mention several of them.

\begin{conj} \label{C3} Let $k$ be a field of characteristic zero. Then  $\GA_n(k)=\textup{GLND}_n(k)$, which is defined as  $<e^{\LND_n(k)}, \textup{GL}_n(k)>$.
\end{conj}


\begin{conj}\label{C4}
$\GA_n(\C)=\textup{GLFD}_n(\C)$, which is defined as $<e^{\textup{LFD}_n(\C)}>$.
\end{conj}

For $k=\C$, conjecture \ref{C4} is different than conjecture \ref{C3},
as $\textup{GLND}_n(\C)\subseteq \textup{GLFD}_n(\C)$ but it is not clear if all exponents of
for example semisimple derivations are in the previous set. In fact, in our opinion, conjecture \ref{C4} is more natural, as it is obvious that $\textup{GLFD}_n(\C)$
is a  normal subgroup, but we do not know if the subgroup $\textup{GLND}_n(\C)$ is normal.

Another one is the following, from \cite{FM07} (where it is stated only for $k=\C$):

\begin{conj} \label{FuMa} Let $k$ be a field.
$\GA_n(k)=\textup{GLF}_n(k)$, where $\textup{GLF}_n(k)$ is the group generated by  all locally finite polynomial automorphisms
(which are polynomial automorphisms $F$ for which the sequence $\{\deg(F^n)\}_{n\in \N}$ is
bounded).
\end{conj}

The subgroup $\textup{GLF}_n(k)$
is normal for any field $k$:
If $F\in \textup{GLF}_n(\C)$, then the sequence $\{\deg(\varphi^{-1}F^m\varphi)\}_{m\in \N}$ is bounded by the bounded sequence
$\{\deg(\varphi^{-1})\deg(F^m)\deg(\varphi)\}_{m\in \N}$.

Then there is the following conjecture, which to our knowledge
originates from Shpilrain in \cite[problem 2, p. 16]{Shp} (there
stated for $k=\C$):

\begin{conj} \label{SHP} $\GA_n(k)=\textup{GSHP}_n(k)$, where $\textup{GSHP}_n(k)=<\GA_{n-1}(k[X_n]),\Aff_n(k)>$, interpreting
$\GA_{n-1}(k[X_n])$ as the automorphisms in $\GA_n(k)$ which fix the last variable.
\end{conj}

He suggests immediately that this conjecture may have
counterexamples in dimension 3 of the form $\exp(D)$ where $D\in
\LND_3(\C)$ which does not have coordinates in its kernel, as
constructed by G. Freudenburg in \cite{Fre97}. Also, it is not
clear if $\textup{GSHP}_n(k)$  is a  normal subgroup of
$\GA_3(k)$.

\subsection{The group $\textup{GLIN}_n(k)$}

Let us denote by $\textup{Lin}_n(k)$ the set of linearizable polynomial automorphisms.
We define $\textup{GLIN}_n(k):=<\textup{Lin}_n(k)>$ as the group generated by the linearizable automorphisms.
This is by construction the smallest normal subgroup of $\GA_n(k)$ containing $\GL_n(k)$.

\begin{lemma} \label{tame}
If $\kar(k)\not =2$, then $\textup{GLIN}_n(k)$ contains $\textup{T}_n(k)$.
\end{lemma}

\begin{proof}
It suffices to show the lemma for an elementary map
$E_{f}:=(X_1+f,X_2,\ldots,X_n)$ where $f\in k[X_2,\ldots,X_n]$. Define $L:=(2X_1,X_2,\ldots,X_n)$ which is in $\GL_n(k)$ as $\kar(k)\not =2$.
The result follows since $E_f=L^{-1} (E_{-2f}LE_{2f})$.
\end{proof}

\begin{remark}
The first author will show in a future preprint that $(X+Y^3,Y)\in
\TA_2(\F_2)$ is not in $\textup{GLIN}_2(\F_2)$.
\end{remark}

\begin{corollary} If $\kar(k)\not =2$, then $\textup{GLIN}_n(k)$ is the smallest normal subgroup of $\GA_n(k)$ containing $\textup{T}_n(k)$.
\end{corollary}

In light of this lemma, {\em and} the result  of theorem \ref{NagataIsInLIN} (being $N\in \textup{GLIN}_n(\C)$),
it is natural to pose the following (as far as we know, new)
conjecture:

\begin{conj} \label{new}  $\textup{GLIN}_n(k)=\GA_n(k)$ (if $\kar(k)\not =2$).
\end{conj}

For $\kar(k)=2$, one might replace $\textup{GLIN}_n(k)$ by the
smallest normal subgroup of $\GA_n(k)$ containing $T_n(k)$. We
remark that for $k=\C$ we have the following chain of inclusions:
\[
\begin{array}{lllll}
&&\textup{GLIN}_n(\C)\\
&\subsetneq&&\subseteq\\
\textup{TA}_n(\C)&&&&\textup{GLFD}_n(\C)\subseteq \textup{GLF}_n(\C)\subseteq \GA_n(\C)\\
&\subsetneq&&\subseteq\\
&&\textup{GLND}_n(\C)\\
\end{array}
 .\]
Any inequality or equality in this chain would be very interesting.
(The set $\textup{GSHP}_n(\C)$ is sort of separate.)
Remark that $\textup{GLFD}_n,\textup{GLF}_n$ and $\textup{GLIN}_n$ are all normal, only the latter two can be defined over any field.


Let us recall the following conjecture from \cite{Mau01,Mau08}:

\begin{conj} \label{finite} Let $F\in \GA_n(\mathbb{F}_q)$. If $q$ is even and $q\not =2$, then only half (the even ones) of the bijections of
$(\mathbb{F}_q)^n\lp (\mathbb{F}_q)^n$
are given by maps in $ \GA_n(\mathbb{F}_q)$.
\end{conj}

Here, we say that a bijection of $(\mathbb{F}_q)^n$is even, if it is even if seen as an element of the permutation group on $q^n$ elements.
In \cite{Mau01}, theorem 2.3, it is concluded that the tame automorphisms over $\mathbb{F}_q$ give all bijections in case $q$ is odd or $q=2$, and only the even bijections in case
$q=4,8,16,\ldots$.

\begin{remark} \label{ZZ}If the conjecture \ref{finite} would {\em not} be true for some $q=2^m, m\geq 2$, this would give a
{\em ridiculously simple counterexample} to the the (already rejected) ``tame generators problem'' for $\F_{q}$.
Also, it will imply that conjecture \ref{SHP} is not true, and
the smallest normal subgroup of $\GA_n(k)$ containing $T_n(k)$ does not equals $\GA_n(k)$ (and en passant conjecture \ref{new} is not true for $k=\F_q$).
\end{remark}

The remark follows from the fact that any conjugate of an even
bijection is again even, and from the fact that any
$F=(F_1(X,Y,Z),F_2(X,Y,Z),Z)\in GA_2(\F_{2^m}[Z]),$  $m\geq 2$, is
even: fix $Z=a\in \F_{2^m}$, and the map
$F_a:=(F_1(X,Y,a),F_2(X,Y,a),a)$ is a tame map on
$\F_{2^m}^2\times \{a\}$ by Jung-van der Kulk-theorem (and hence
even because of theorem 2.3 in \cite{Mau01}).

\ \\
\ \\
{\bf Acknowledgements:}
The authors would like to thank the staff of the Mathematisches Forschungsinistitut Oberwolfach for their hospitality.


\begin{thebibliography}{00}

\bibitem{AMS} AMS featured article review by G. H. Meisters on the paper \cite{CEGHM94},\\ http://www.ams.org/mathscinet-getitem$?$mr=1483974


 \bibitem{A83} V. I. Arnold,
Geometrical methods in the theory of ordinary differential equations,
{\em Grundlehren der Mathematischen Wissenschaften} [Fundamental Principles of Mathematical Sciences], 250. Springer-Verlag, New York, 1988.

\bibitem{B84} H. Bass, {\em
A nontriangular action of $G\sb{a}$ on $A\sp{3}$},
J. Pure Appl. Algebra 33 (1984), no. 1, 1-5.


\bibitem{CEGHM94}A. Cima, A.  van den Essen, A. Gasull, E. Hubbers, F.  Mañosas,
{\em A polynomial counterexample to the Markus-Yamabe conjecture},
Adv. Math. 131 (1997), no. 2, 453--457.

\bibitem{DMZ95} B. Deng, G. H. Meisters, G.  Zampieri,
{\em Conjugation for polynomial mappings},
Z. Angew. Math. Phys. 46 (1995), no. 6, 872--882.

\bibitem{Essenboek} A. van den Essen, \emph{Polynomial Automorphisms and the
Jacobian Conjecture,} volume 190 of \emph{in Progress in Mahtematics},
Birkh\"auser (2000)

\bibitem{E} A. van den Essen, {\em
A counterexample to a conjecture of Meisters}, Automorphisms of affine spaces (Curaçao, 1994), 231--233, Kluwer Acad. Publ., Dordrecht, 1995.

\bibitem{EH96} A. van den Essen, E. Hubbers, {\em
Polynomial maps with strongly nilpotent Jacobian matrix and the Jacobian conjecture},
Linear Algebra Appl. 247 (1996), 121--132.


\bibitem{Fre06} G. Freudenburg, Algebraic theory of locally nilpotent derivations.
Encyclopaedia of Mathematical Sciences, 136.
Invariant Theory and Algebraic Transformation Groups, VII. Springer-Verlag, Berlin, 2006.
Verlag.

\bibitem{Fre97}  G.  Freudenburg, {\em Local slice construction in $k[X,Y,Z]$}, Osaka J. Math 34 (1997), 757-767

\bibitem{FM07}  J-Ph. Furter and S. Maubach,  {\em  Locally finite polynomial endomorphisms},
J. Pure Appl. Algebra 211 (2007), no. 2

\bibitem{FM08} J-Ph. Furter and S. Maubach, {\em Semisimple plane automorphisms}, submitted. http://arxiv.org/abs/0804.2157


\bibitem{Shp} J. Guitierrez, V. Shpilrain, and J-T. Yu, {\em   Affine algebraic geometry},  Contemporary Mathematics {369} (2004), 1-30 .

\bibitem{Jung} H. Jung, {\em \"Uber ganze birationale Transformationen der Ebene},  (German)
J. Reine Angew. Math. 184, (1942). 161--174.

\bibitem{Kulk} W. van der Kulk, {\em
On polynomial rings in two variables}
Nieuw Arch. Wiskunde (3) 1, (1953). 33--41.

 \bibitem{Mau00} S. Maubach, {\em Triangular monomial derivations on $k[X_1,X_2,X_3,X_4]$ have kernel generated by at most four elements},
 J.Pure Appl. Algebra 153 (2000) no.2. 165-170

\bibitem{Mau01} S. Maubach,{\em      Polynomial automorphisms over finite fields},  Serdica Math. J. 27 (2001) no.4. 343-350.

\bibitem{Mau00a} S. Maubach, {\em  An algorithm to compute the kernel of a derivation up to a certain degree},  J.Symbolic Computation 29 (2000), no.6. 959-970

\bibitem{Mau08} S. Maubach, {\em A problem on polynomial maps over finite fields}, unpublished, arXiv:0802.0630v1

\bibitem{MP07} S. Maubach and  H. Peters, {\em   Polynomial maps which are roots of power series},   accepted to Mathematische Zeitschrift.


\bibitem{N72} M. Nagata, {\em On automorphism group of $k[x,\,y]$},
Department of Mathematics, Kyoto University, Lectures in Mathematics, No. 5. Kinokuniya Book-Store Co., Ltd., Tokyo, 1972.

\bibitem{SU04} I. Shestakov and U. Umirbaev, {\em The tame and the wild automorphisms of polynomial rings in three variables},
 J. Amer. Math. Soc.  17  (2004),  no. 1, 197--227

\bibitem{SU04a} I. Shestakov and U. Umirbaev, {\em Poisson brackets and two-generated subalgebras of rings of polynomials},
J. Amer. Math. Soc. 17 (2004), no. 1, 181--196












































\bibitem{vdE}
A. van den Essen, Polynomial automorphisms and the {J}acobian
conjecture, Birkh\"auser Verlag, Basel, 2000.







\end{thebibliography}
\end{document}